\documentclass[12pt,a4paper]{article}

\usepackage{t1enc}

\usepackage{amsmath}
\usepackage{amssymb}
\usepackage{amscd}
\usepackage{amsthm}
\usepackage{amsfonts}

\usepackage{graphics}
\usepackage{graphicx}
\usepackage{booktabs}

\usepackage{url}
\usepackage[colorlinks=true,citecolor=black,linkcolor=black,urlcolor=blue]{hyperref}

\newcommand{\doi}[1]{\href{http://dx.doi.org/#1}{\texttt{doi:#1}}}

\theoremstyle{plain}
\newtheorem{thm}{Theorem}
\newtheorem{cor}[thm]{Corollary}
\newtheorem{lemma}[thm]{Lemma}
\newtheorem{prop}[thm]{Proposition}
\newtheorem{conj}[thm]{Conjecture}

\theoremstyle{definition}
\newtheorem{rem}[thm]{Remark}

\newtheorem*{defi}{Definition}
\newtheorem*{nota}{Notation}
\newtheorem*{ex}{Example}

\newcommand{\qwy}{\textbf}

\newcommand{\uf}{\mathbb{F}}

\newcommand{\hb}{\mathcal{B}}
\newcommand{\hd}{\mathcal{D}}
\newcommand{\hu}{\mathcal{U}}

\newcommand{\ba}{\mathbf{a}}
\newcommand{\bz}{\mathbf{z}}

\renewcommand{\div}{\quad\mathrm{div}\;}

\DeclareMathOperator{\GF}{\mathsf{GF}} 
\DeclareMathOperator{\PG}{\mathsf{PG}}
\DeclareMathOperator{\AG}{\mathsf{AG}}

\DeclareMathOperator{\spa}{Span}

\DeclareMathOperator{\kar}{char}

\begin{document}

\title{\bf The number of directions determined by less than $q$ points}
\author{Szabolcs~L.~Fancsali, P\'eter~Sziklai and Marcella~Tak\'ats\footnote{The authors were partially supported by the following grants: OTKA K 81310 and OTKA CNK 77780, ERC, Bolyai and T\'AMOP 4.2.1./B-09/KMR-2010-0003}}
\date{February 28, 2012}

\maketitle

{\small\noindent\MakeUppercase{Notice:} this is the author's version of a work that was submitted for publication in \emph{J. Algebr. Combin.} Changes resulting from the publishing process, such as peer review, editing, corrections, structural formatting, and other quality control mechanisms are not reflected in this document. A definitive version was subsequently published in \emph{J. Algebr. Combin.}, Volume~\textbf{37}, Issue~1, (2013), Pages 27--37. The final publication is available at Springer via \doi{10.1007/s10801-012-0357-1}}\\

\begin{abstract}
In this article we prove a theorem about the number of directions
determined by less then $q$ affine points, similar to the result of
Blokhuis~et.~al.~\cite{bbbss} on the number of directions determined
by $q$ affine points.
\end{abstract}

\section{Introduction}

In this article, $p$ is a prime and $q=p^h$, where $h\ge 1$. $\GF(q)$
denotes the finite field with $q$ elements, and $\uf$ can denote an
arbitrary field (or maybe a Euclidean ring). $\PG(d,q)$ denotes the
projective geometry of dimension $d$ over the finite field
$\GF(q)$. $\AG(d,q)$ denotes the affine geometry of dimension $d$ over
$\GF(q)$ that corresponds to the co-ordinate space $\GF(q)^d$ of rank
$d$ over $\GF(q)$.

For the affine and projective planes $\AG(2,q)\subset\PG(2,q)$, we
imagine the line $\ell_\infty=\PG(2,q)\setminus\AG(2,q)$ at infinity
as the set $\ell_\infty\cong\GF(q)\cup\{\infty\}$. So the
\emph{non-vertical directions} are field-elements (numbers) and the
\emph{vertical direction} is $\infty$.

The original problem of direction sets was the following. Let
$f:\GF(q)\rightarrow\GF(q)$ be a function and let
$\hu=\{(x,f(x))\;|\;x\in\GF(q)\}\subseteq\AG(2,q)$ be the graph of the
function $f$. The question is, how many \emph{directions} can be
\emph{determined} by the graph of $f$.
\begin{defi}[Direction set]
If $\hu$ denotes an arbitrary set of points in the affine plane
$\AG(2,q)$ then we say that the set
\begin{equation*}
\hd=\left\{\left.\frac{b-d}{a-c}\;\right|
\;(a,b),(c,d)\in\hu,(a,b)\neq(c,d)\right\}
\end{equation*}
is the \emph{set of directions determined by $\hu$}. We define
$\dfrac{a}{0}$ as $\infty$ if $a\neq 0$, thus
$\hd\subseteq\GF(q)\cup\{\infty\}$. If $\hu$ is the graph of a
function, then it simply means that $|\hu|=q$ and $\infty\notin\hd$.
\end{defi}

In~\cite{Ball2003}, \textrm{Simeon Ball} proved a stronger version of
the structure theorem of \textrm{Aart Blokhuis}, \textrm{Simeon Ball},
\textrm{Andries Brouwer}, \textrm{Leo Storme} and
\textrm{Tam\'as Sz\H{o}nyi}, published in~\cite{bbbss}.
To recall their result we need some definitions.
\begin{defi}
Let $\hu$ be a set of points of $\AG(2,q)$. If $y\in\ell_\infty$ is an
arbitrary direction, then let $s(y)$ denote the greatest power of $p$
such that each line $\ell$ of direction $y$ meets $\hu$ in zero modulo
$s(y)$ points. In other words,
\begin{equation*}
s(y)=\gcd\big(\left\{|\ell\cap\hu|\;\big|\;\ell\cap\ell_\infty=\{y\}\right\}\cup\left\{p^h\right\}\big).
\end{equation*}
Let $s$ be the greatest power of $p$ such that each line $\ell$ of
direction in $\hd$ meets $\hu$ in zero modulo $s$ points. In other
words,
\begin{equation*}
s=\gcd_{y\in\hd}s(y)=\min_{y\in\hd}s(y).
\end{equation*}
Note that $s(y)$ and thus also $s$ might be equal to $1$.
Note that $s(y)=1$ for each non-determined direction $y\notin\hd$.
\end{defi}
\begin{rem}\label{rem:1mods}
Suppose that $s\ge p$. Then for \emph{each} line
$\ell\subset\PG(2,q)$:
\begin{align*}
\quad\quad\quad
&\mbox{\textbf{either}} &
\left(\hu\cup\hd\right)\cap\ell\;&=\emptyset;
\quad\quad\quad\\
\quad\quad\quad
&\mbox{\textbf{or}} &
\left|\left(\hu\cup\hd\right)\cap\ell\right|&\equiv 1\pmod{s}.
\quad\quad\quad
\end{align*}
Moreover, $|\hu|\equiv 0\pmod{s}$.
\end{rem}

\noindent
(If $s=1$ then $0\equiv 1\pmod{s}$, so in this case these remarks above
would be meaningless.)

\begin{proof}
Fix a direction $y\in\hd$. Each affine line with slope $y$ meets $\hu$
in zero modulo $s$ points, so $|\hu|\equiv 0\pmod{s}$.

\noindent
An affine line $L\subset\AG(2,q)$ with slope $y\in\hd$
meets $\hu$ in $0\pmod{s}$ points, so the \emph{projective} line
$\ell=L\cup\!\{y\}$ meets $\hu\cup\hd$ in $1\pmod{s}$ points.

\noindent
An affine line $L\subset\AG(2,q)$ with slope $y\notin\hd$ meets
$\hu$ in at most one point, so the \emph{projective} line
$\ell=L\cup\!\{y\}$ meets $\hu\cup\hd$ in either zero or one point.

\noindent
Let $P\in\hu$ and let $L\subset\AG(2,q)$ an
affine line with slope $y\in\hd$, such that $P\in L$. Then the
\emph{projective} line $\ell=L\cup\!\{y\}$ meets $\hu$ in
$0\pmod{s}$ points, and thus, $\ell$ meets $\hd\cup\hu\setminus\{P\}$
also in $0\pmod{s}$ points. Thus, considering all the lines through
$P$ (with slope in $\hd$), we get $|\hu\cup\hd|\equiv 1\pmod{s}$.
Since $\hu$ has $0\pmod{s}$ points, $|\hd|\equiv 1\pmod{s}$. So we get
that $\hu\cup\hd$ meets also the ideal line in $1\pmod{s}$ points.
\end{proof}

\begin{rem}[Blocking set of R\'edei type]\label{rem:mbsoRt}
If $|\hu|=q$ then each of the $q$ affine lines with slope $y\notin\hd$
meets $\hu$ in exactly one point, so $\hb=\hu\cup\hd$ is a
blocking set meeting each projective line in $1\pmod{s}$
points. Moreover, if $\infty\notin\hd$ then $\hu$ is the graph of a
function, and in this case the blocking set $\hb$ above is called
\emph{of R\'edei type}.
\end{rem}

\begin{thm}[Blokhuis, Ball, Brouwer, Storme and Sz\H{o}nyi;
and Ball]\label{thm:ball}
{\rm \cite{bbbss} and \cite[Theorem~1.1]{Ball2003}}
Let $|\hu|=q$ and $\infty\notin\hd$. Using the notation $s$ defined
above, one of the following holds:
\begin{align*}
\quad
&\mbox{\textbf{\emph{either}}}\quad\quad\quad\quad\; s=1
& &\mbox{and}&
\dfrac{q+3}{2}\le\left|\hd\right|&\le q;\quad\\
\quad
&\mbox{\textbf{\emph{or}}}\quad\GF(s)\mbox{ is a subfield of }\GF(q)
& &\mbox{and}&
\dfrac{q}{s}+1\le\left|\hd\right|&\le\dfrac{q-1}{s-1};\quad\\
\quad
&\mbox{\textbf{\emph{or}}}\quad\quad\quad\quad\quad\quad s=q
& &\mbox{and}& |\hd|&=1.\quad
\end{align*}
Moreover, if $s>2$ then $\hu$ is a $\GF(s)$-linear affine
set (of rank $\log_s q$).\qed
\end{thm}
\begin{defi}[Affine linear set]
A \emph{$\GF(s)$-linear affine set} is the $\GF(s)$-linear span of
some vectors in
$\AG(n,q)\cong\GF(s^{\log_s q})^n\cong\GF(s)^{n\log_s q}$ (or possibly
a translate of such a span). The \emph{rank} of the affine linear set
is the rank of this span over $\GF(s)$.
\end{defi}
What about the directions determined by an affine set
$\hu\subseteq\AG(2,q)$ of cardinality \emph{not}~$q$?
Using the pigeon hole principle, one can easily prove
that if $|\hu|>q$ then it determines all the $q+1$ directions.
So we can restrict our research to affine sets of
\emph{less than} $q$ points.

Examining the case $q=p$ prime,
\textrm{Tam\'as Sz\H{o}nyi}~\cite{Szonyi1999} and later
(independently) also \textrm{Aart Blokhuis}~\cite{Blokhuis2001} have
proved the following result.
\begin{thm}[Sz\H{o}nyi; Blokhuis]\label{thm:sztaab}
{\rm \cite[Theorem~5.2]{Szonyi1999}}
Let $q=p$ prime and suppose that $1<|\hu|\le p$. Also suppose that
$\infty\notin\hd$. Then
\begin{align*}
\quad\quad\quad\quad\quad\quad\quad
&\mbox{\textbf{\emph{either}}}
&\dfrac{|\hu|+3}{2}\le\left|\hd\right|&\le p;
\quad\quad\quad\quad\quad\quad\quad\\
\quad\quad\quad\quad\quad\quad\quad
&\mbox{\textbf{\emph{or}}}\quad\hu\mbox{ is collinear}
&\mbox{ and }\quad\hphantom{\le}\left|\hd\right|&=1.
\quad\quad\quad\quad\quad\quad\quad
\end{align*}
Moreover, these bounds are sharp.
\qed
\end{thm}

In this article we try to generalize this result to the $q=p^h$ prime
\emph{power} case by proving an analogue of
\qwy{Theorem~\ref{thm:ball}} for the case $|\hu|\le q$. Before we
examine the number of directions determined by less than $q$ affine
points in the plane, we ascend from the plane in the next section and
examine the connection between linear sets and direction sets in
arbitrary dimensions. The further sections will return to the plane.

\section{Linear sets as direction sets}

The affine space $\AG(n,q)$ and its ideal hyperplane
$\Pi_\infty\cong\PG(n\!-\!1,q)$ of directions together constitute a
projective space $\PG(n,q)$. We say that the point $P\in\Pi_\infty$ is
a direction determined by the affine set $\hu\subset\AG(n,q)$ if there
exists at least one line through $P$ that meets $\hu$ in at least two
points.
\begin{defi}[Projective linear set]
Suppose that $\GF(s)$ is a subfield of $\GF(q)$. A
\emph{projective $\GF(s)$-linear set} $\hb$ of rank \mbox{$d+1$} is
a projected image of the canonical subgeometry
$\PG(d,s)\subset\PG(d,q)$ from a center disjoint to this
subgeometry. The projection can yield multiple points.
\end{defi}
\begin{prop}
Suppose that $\hu$ is an affine $\GF(s)$-linear set of rank $d+1$ in
$\AG(n,q)$ such that $\AG(n,q)$ is the smallest dimensional
affine subspace that contains $\hu$. Let $\hd$ denote the set of
directions determined by $\hu$. The set $\hu\cup\hd$ is a projective
$\GF(s)$-linear set of rank $d+1$ in $\PG(n,q)$ and all the multiple
points are in $\hd$.
\end{prop}
\begin{proof}
Without loss of generality, we can suppose that $\hu$ contains the
origin and suppose that $\hu$ is the set of $\GF(s)$-linear
combinations of the vectors $\ba_0,\ba_1,\dots,\ba_d$. We can
coordinatize $\AG(n,q)$ such that $\ba_{d-n+1},\dots,\ba_d$ is the
standard basis of $\GF(q)^n\cong\AG(n,q)$.

\noindent
Embed $\GF(q)^n\cong\AG(n,q)$ into $\GF(q)^{d+1}\cong\AG(d\!+\!1,q)$
such that
$\bz_0,\bz_1,\dots,\linebreak[4]\bz_{d-n},\ba_{d-n+1},\dots,\ba_d$ is
the standard basis. Let $\pi$ denote the projection of
$\AG(d\!+\!1,q)$ onto $\AG(n,q)$ such that $\pi(\bz_i)=\ba_i$ for each
$i=0,\dots,d-n$ and $\pi(\ba_j)=\ba_j$ for each $j>d-n$. Then $\hu$ is
the image of the canonical subgeometry $\AG(d\!+\!1,s)$ by $\pi$.

\noindent
Extend $\pi$ to the ideal hyperplane. The extended $\bar{\pi}$ is a
collineation so the image of a determined direction is a determined
direction, and vice versa, let $A$ and $B$ two arbitrary distinct
points in $\ell\cap\hu$ and let $P$ be the direction determined by
$\pi^{-1}(A)$ and $\pi^{-1}(B)$. Then the direction of $\ell$ is
$\bar{\pi}(P)$.
\end{proof}
\begin{cor}
If $\hd$ is the set of directions determined by an affine
$\GF(s)$-linear set, then $\hd$ is a projective $\GF(s)$-linear
set.\qed
\end{cor}
\begin{rem}
In~\cite[Proposition~2.2]{polverino2009}, Olga Polverino proved that
if $\hd$ is a projective $\GF(s)$-linear set then
$|\hd|\equiv 1\pmod{s}$.
\qed
\end{rem}

The proposition above says that the set of directions determined by an
affine linear set is a projective linear set. The converse of this
proposition is also true; each projective linear set is a direction
set:
\begin{thm}
Embed $\PG(n,q)$ into $\PG(n\!+\!1,q)$ as the ideal hyperplane and let
$\AG(n\!+\!1,q)=\PG(n\!+\!1,q)\setminus\PG(n,q)$ denote the affine
part. \emph{For each} projective $\GF(s)$-linear set $\hd$ of rank
$d+1$ in $\PG(n,q)$, there exists an \emph{affine} $\GF(q)$-linear set
$\hu$ of rank $d+1$ in $\AG(n\!+\!1,q)$ such that the set of
directions determined by $\hu$ is $\hd$.
\end{thm}
\begin{proof}
$\hd\subset\PG(n,q)$ is the image of the canonical subgeometry
$\PG(d,s)\subset\PG(d,q)$ by the projection
$\pi:\PG(d,q)\rightarrow\PG(n,q)$ where the center $C$ of $\pi$ is
disjoint to this subgeometry. Embed $\PG(d,q)$ into $\PG(d\!+\!1,q)$
as the ideal hyperplane and extend $\pi$ to
$\bar{\pi}:\PG(d\!+\!1,q)\rightarrow\PG(n\!+\!1,q)$ such that its
center remains $C$. That is, the center is in the ideal
hyperplane. Consider the canonical subgeometry
$\PG(d\!+\!1,s)\subset\PG(d\!+\!1,q)$ and its image by $\bar{\pi}$.
\begin{equation*}
\begin{CD}
\PG(d\!+\!1,s)@>\subset>>\PG(d\!+\!1,q)@>\bar{\pi}>>\PG(n\!+\!1,q)\\
@VVV@VVV@VVV\\
\PG(d,s)@>\subset>>\PG(d,q)@>\pi>>\PG(n,q)
\end{CD}
\end{equation*}
The `ideal part' of this canonical subgeometry $\PG(d,s)$ is the
original canonical subgeometry $\PG(d,s)$ of $\PG(d,q)$ and the
projection $\bar{\pi}$ project this onto $\hd$. Since the center is
totally contained in the ideal hyperplane, $\bar{\pi}$ maps the affine
part of the canonical subgeometry $\PG(d\!+\!1,s)$ one-to-one.

\noindent
The directions determined by the affine part of $\PG(d\!+\!1,s)$ are
the points of $\PG(d,s)$ in the ideal hyperplane of
$\AG(d\!+\!1,q)$. Since the extended $\bar{\pi}$ pre\-serves
collinearity, the set of directions determined by the projected image
of the affine part is $\hd$.
\end{proof}

\section{The R\'edei polynomial of less than {\boldmath $q$} points}

Let $\hu$ be a set of less than $q$ affine points in $\AG(2,q)$ and
let $\hd$ denote the set of directions determined by $\hu$. Let
$n=|\hu|$ and let $R(X,Y)$ be the inhomogeneous affine R\'edei
polynomial of the affine set $\hu$, that is,
\begin{equation*}
R(X,Y)=\prod_{(a,b)\in\hu}(X-aY+b)=X^n+\sum_{i=0}^{n-1}\sigma_{n-i}(Y)X^{i}
\end{equation*}
where the abbreviation $\sigma_k(Y)$ means the $k$-th elementary
symmetric polynomial of the set $\{b-aY\;|\;(a,b)\in\hu\}$ of linear
polynomials.
\begin{prop}
If $y\in\hd$ then
$R(X,y)\in\GF(q)[X^{s(y)}]\setminus\GF(q)[X^{p\cdot s(y)}]$.

\noindent
If $y\notin\hd$ then $R(X,y)\mid X^q-X$.
\end{prop}
\begin{proof}
Assume $y\in\hd$. Then the equation $R(X,y)=0$ has root $x$ with
multiplicity $m$ if there is a line with slope $y$ meeting $\hu$ in
exactly $m$ points. The value of $x$ determines this line. So each $x$
is either not a root of $R(X,y)$ or a root with multiplicity a
multiple of $s(y)$; and $p\cdot s(y)$ does not have this
property. Since $R$ is totally reducible, it is the product of its
root factors.

\noindent
If $y\notin\hd$ then a line with direction $y$ cannot meet $\hu$ in
more than one point, so an $x$ cannot be a multiple root of $R(X,y)$.
\end{proof}

\begin{nota}
Let $\uf$ be the polynomial ring $\GF(q)[Y]$ and consider $R(X,Y)$ as
the element of the univariate polynomial ring $\uf[X]$. Divide $X^q-X$
by $R(X,Y)$ as a univariate polynomial over $\uf$ and let $Q$ denote
the quotient and let $H+X$ be the negative of the remainder.
\begin{align*}
Q(X,Y)&=(X^q-X)\div R(X,Y) &\mbox{over}\quad\uf\\
-X-H(X,Y)&\equiv(X^q-X)\mod R(X,Y) &\mbox{over}\quad\uf
\end{align*}
So
\begin{equation*}
R(X,Y)Q(X,Y)=X^q+H(X,Y)=X^q+\sum_{i=0}^{q-1}h_{q-i}(Y)X^i
\end{equation*}
where $\deg_XH<\deg_XR$.
Let $\sigma^*$ denote the coefficients of $Q$,
\begin{equation*}
Q(X,Y)=X^{q-n}+\sum_{i=0}^{q-n-1}\sigma^*_{q-n-i}(Y)X^{i}
\end{equation*}
and so
\begin{equation*}
h_j(Y)=\sum_{i=0}^{j}\sigma_i(Y)\sigma^*_{j-i}(Y).
\end{equation*}
We know that $\deg h_i\le i$, $\deg\sigma_i\le i$ and
$\deg\sigma^*_i\le i$. Note that the $\sigma^*(Y)$ polynomials are not
necessarily elementary symmetric polynomials of linear polynomials
and if $y\in\hd$ then $Q(X,y)$ is not necessarily totally
reducible.
\end{nota}
\begin{rem}
Since $\deg_X H<\deg_X R$, we have $h_i= 0$ for $1\le i\le q-n$.
By definition, $\sigma_0=\sigma^*_0=1$. The equation $h_1=0$
implies $\sigma^*_1=-\sigma_1$, this fact and the equation $h_2=0$
implies $\sigma^*_2=-\sigma_2+\sigma_1^2$ and so on, the $q-n$
equations $h_i=0$ uniquely define all the coefficients
$\sigma^*_i$.
\end{rem}
\begin{prop}\label{prop:es}
If $y\in\hd$ then
$Q(X,y),H(X,y)\in\GF(q)[X^{s(y)}]$ and if $\deg R\le\deg Q$ then
$Q(X,y)\in\GF(q)[X^{s(y)}]\setminus\GF(q)[X^{p\cdot s(y)}]$.

\noindent
If $y\notin\hd$ then $R(X,y)Q(X,y)=X^q+H(X,y)=X^q-X$. In this case
$Q(X,y)$ is also a totally reducible polynomial.
\end{prop}
\begin{proof}
If $y\in\hd$ then
\begin{equation*}
R(X,y)=X^n+\sum_{i=0}^{n-1}\sigma_{n-i}(y)X^{i}
\in\GF(q)[X^{s(y)}]\setminus\GF(q)[X^{p\cdot s(y)}].
\end{equation*}
So $s(y)\mid n$ and
$\sigma_{i}(y)\neq 0\Rightarrow s(y)\mid n-i\Rightarrow s(y)\mid i$.
The defining equation of $\sigma^*_i$ contains the sum of products of
some $\sigma_j$ where the sum of indices (counted with multiplicities)
is $i$. Since $\sigma_{j}(y)\neq 0$ only if $s(y)\mid j$, also
$\sigma^*_{i}(y)\neq 0$ only if $s(y)\mid i$.

\noindent
If $\deg R\le\deg Q$ then we can consider $R$ as $(X^q-X)\div Q$ and
the reminder is the same $H$.

\noindent
Since both $R(X,y)$ and $Q(X,y)$ are in $\GF(q)[X^{s(y)}]$,
$H(X,y)\in\GF(q)[X^{s(y)}]$.

\noindent
If $y\notin\hd$ then $R(X,y)\mid(X^q-X)$ in $\GF(q)[X]$ so $Q(X,y)$ is
also totally reducible.
\end{proof}
\begin{rem}
Note that $H(X,y)$ can be an element of $\GF(q)[X^{p\cdot s(y)}]$.
If $H(X,y)\equiv a$ is a constant polynomial, then
$R(X,y)Q(X,y)=X^q+a=X^q+a^q=(X+a)^q$. This means that $R(X,y)=(X+a)^n$
and thus, there exists exactly one line (corresponding to $X=-a$) of
direction $y$ that contains $\hu$, and so $\hd=\{y\}$.
\end{rem}
\begin{defi}
If $|\hd|\ge 2$ (i.e. $H(X,y)$ is not a constant polynomial) then for
each $y\in\hd$, let $t(y)$ denote the maximal power of $p$ such that
$H(X,y)=f_y(X)^{t(y)}$ for some $f_y(X)\notin\GF(q)[X^p]$.
\begin{equation*}
H(X,y)\in\GF(q)[X^{t(y)}]\setminus\GF(q)[X^{t(y)p}].
\end{equation*}
In this case $t(y)<q$ since $t(y)\le\deg_X H<q$.
Let $t$ be the greatest common divisor of the numbers $t(y)$, that is,
\begin{equation*}
t=\gcd_{y\in\hd}t(y)=\min_{y\in\hd}t(y).
\end{equation*}
If $H(X,y)\equiv a$ (i.e. $\hd=\{y\}$) then we define $t=t(y)=q$.
\end{defi}
\begin{rem}
If there exists at least one determined direction $y\in\hd$ such that
$H(X,y)$ is not constant then $t<q$. From
\qwy{Proposition~\ref{prop:es}} we have
$s(y)\le t(y)$ for all $y\in\hd$, so $s\le t$.\qed
\end{rem}
\begin{prop}\label{prop:lin}
Using the notation above,
\begin{equation*}
R(X,Y)Q(X,Y)=X^q+H(X,Y)\in\spa_\uf\langle
1,X,X^{t},X^{2t},X^{3t},\dots,X^{q}\rangle.
\end{equation*}
\end{prop}
\begin{proof}
If $|\hd|=\{y\}$ then $H(X,y)\equiv a$ and $H(X,z)=-X$ for $z\neq y$.

\noindent
Suppose that $|\hd|\ge 2$. If $y\notin\hd$ then $X^q+H(X,y)=X^q-X$ and
if $y\in\hd$ then
$X^q+H(X,y)\in\GF(q)[X^{t(y)}]\setminus\GF(q)[X^{t(y)p}]$.

\noindent
Thus, in both cases, if $i\neq 1$ and $i\nmid t$, then $h_{q-i}(Y)$ has
$q$ roots and its degree is less than $q$.
\end{proof}

\section{Bounds on the number of directions}

Although, in the original problem, the vertical direction $\infty$ was
not determined, from now on, without loss of generality we suppose
that $\infty$ is a determined direction (if not, we apply an affine
collineation). We continue to suppose that there is at least one
non-determined direction.
\begin{lemma}\label{lem:egyes}
If $\infty\in\hd\subsetneqq\ell_\infty$ then
$|\hd|\ge\deg_X H(X,Y)+1$.
\end{lemma}
\begin{proof}
If $y\notin\hd$ then $R(X,y)\mid X^q-X$, thus $H(X,y)=-X$ and thus
$\forall i\neq q-1$: $h_i(y)=0$.

\noindent
If $y\in\hd$ then $R(X,y)\nmid X^q-X$, hence $\exists i\neq q-1$:
$h_i(y)\neq 0$ and thus $h_i\not\equiv 0$. Let $i$ be the smallest
index such that $h_i\not\equiv 0$ and so $i=q-\deg_X H$. Since
$h_i\not\equiv 0$ has at least $(q+1)-|\hd|=q-(|\hd|-1)$ roots,
$\deg_Yh_i\ge q-|\hd|+1$.
\begin{equation*}
\hfill\mbox{Now }
q\ge\deg X^{q-i}h_i(Y)=q-i+\deg_Yh_i\ge 2q-|\hd|+1-i.
\hfill
\end{equation*}
Hence $|\hd|\ge q+1-i=\deg_X H+1$. 
\end{proof}

\begin{lemma}\label{lem:kettes}
Let $\kappa(y)$ denote the number of the roots of $X^q+H(X,y)$ in
$\GF(q)$, counted with multiplicity.
If $X^q+H(X,y)\neq X^q-X$ and if $H(X,y)$ is not a constant
polynomial, then
\begin{equation*}
\frac{\kappa(y)-1}{t(y)+1}+1=
\frac{\kappa(y)+t(y)}{t(y)+1}\le
t(y)\cdot\deg f_y(X)=\deg_X H \le \deg H
\end{equation*}
\end{lemma}
\begin{proof}
Fix $y\in\hd$ and utilize that $X^q+H(X,y)\in\GF(q)[X^{t(y)}]$, thus
\begin{equation*}
\Big(X^{q/t(y)}+f_y(X)\Big)^{t(y)}=X^q+H(X,y)=\Big(a(X)\cdot b(X)\cdot c(X)\Big)^{t(y)}
\end{equation*}
where the totally reducible $a(X)$ contains all the roots (in
$\GF(q)$) without multiplicity, the totally reducible $b(X)$ contains
the further roots (in $\GF(q)$), and $c(X)$ has no root in $\GF(q)$.
(Note that $t(y)<q$ so $X^{q/t(y)}\in\GF(q)[X^p]$.)
\begin{gather*}
\left.
\begin{aligned}
a(X) &\mid (X^q-X)\\
a(X) &\mid \left(X^q+f_y(X)^{t(y)}\right)
\end{aligned}
\right\}
\begin{aligned}[t]
\Longrightarrow
a(X) &\mid\left(f_y(X)^{t(y)}+X\right)\\
b(X) &\mid\partial_X\left(X^{q/t(y)}+f_y(X)\right)=\partial_X f_y(X)\\
a(X)b(X) &\mid \left(f_y(X)^{t(y)}+X\right)\partial_X f_y(X)
\end{aligned}
\end{gather*}
And so,
$\deg(a(X)b(X))\le t(y)\cdot\deg f_y +\deg f_y -1=(t(y)+1)\cdot\deg f_y-1$,
since $\partial_X f_y(X)\neq 0$ and $f_y(X)^{t(y)}=H(X,y)\neq -X$.
We get
\begin{equation*}
\frac{\kappa(y)+t(y)}{t(y)+1}=
t(y)\frac{\deg(a(X)b(X))+1}{t(y)+1}\le
t(y)\cdot\deg f_y(X)
\end{equation*}
using $\kappa(y)=t(y)\cdot\deg(a(X)b(X))$.
\end{proof}

Using these lemmas above we can prove a theorem similar to
\qwy{Theorem~\ref{thm:ball}} but it is weaker in our case.
\begin{thm}\label{thm:m}
Let $\hu\subset\AG(2,q)$ be an arbitrary set of points and let $\hd$
denote the directions determined by $\hu$. We use the notation $s$ and
$t$ defined above geometrically and algebraically,
respectively. Suppose that $\infty\in\hd$. One of the following holds:
\begin{align*}
&\mbox{\textbf{\emph{either}}}& 1=s\le t&<q\quad\mbox{and}&
\dfrac{|\hu|-1}{t+1}+2\le\left|\hd\right|&\le q+1;\\
&\mbox{\textbf{\emph{or}}}& 1<s\le t&<q\quad\mbox{and}&
\dfrac{|\hu|-1}{t+1}+2\le\left|\hd\right|&\le\dfrac{|\hu|-1}{s-1};\\
&\mbox{\textbf{\emph{or}}}& 1\le s\le t&=q\quad\mbox{and}&\hd\;&=\{\infty\}.
\end{align*}
\end{thm}
\begin{proof}
The third case is trivial ($t=q$ means $|\hd|=1$, by the
definition of $t$).

\noindent
Let $P$ be a point of $\hu$ and consider the lines connecting $P$ and
the ideal points of $\hd$. Since each such line meets $\hu$ and has a
direction determined by $\hu$, it is incident with $\hu$ in a multiple
of $s$ points. If $s>1$ then counting the points of $\hu$ on these
lines we get the upper bound.

\noindent
If $t<q$ then we can choose a direction $y\in\hd$ such that the
conditions of \qwy{Lemma~\ref{lem:kettes}} hold.
Using \qwy{Lemma~\ref{lem:egyes}} and
\qwy{Lemma~\ref{lem:kettes}}, we get
\begin{equation*}
|\hd|\ge\deg_X H(X,Y)+1\ge\dfrac{\kappa(y)-1}{t(y)+1}+1+1.
\end{equation*}
The number of roots of $R(X,y)Q(X,y)$ is at least the number of
roots of $R(X,y)$ which equals to $|\hu|$. Thus
$\kappa(y)\ge|\hu|$. And thus
\begin{equation*}
\dfrac{\kappa(y)-1}{t(y)+1}\ge\dfrac{|\hu|-1}{t+1}.
\qedhere
\end{equation*}
\end{proof}
An affine collineation converts Sz\H{o}nyi's and Blokhuis'
\qwy{Theorem~\ref{thm:sztaab}} to the special case
of our \qwy{Theorem~\ref{thm:m}}, since $t$ is equal to
either $1$ or $p$ in the case $q=p$ prime.

In the case $q>p$, the main problem of
\qwy{Theorem~\ref{thm:m}} is, that the
definition of $t$ is non-geometrical. Unfortunately, $t=s$
does not hold in general. For example, let $\hu$ be a
$\GF(p)$-linear set minus one point. In this case $s=1$, but
$t=p$. In the rest of this article, we try to describe this problem.

\section{Maximal affine sets}

One can easily show that a \emph{proper} subset of the affine set
$\hu$ can determine the same directions. (For example, let $\hu$ be
an affine subplane over the subfield $\GF(s)$. Arbitrary $s+1$ points
of $\hu$ determine the same directions.) 
\begin{defi}[Maximal affine set]
We say that $\hu\subseteq\AG(2,q)$ is a \emph{maximal} affine set that
determines the set $\hd\subseteq\ell_{\infty}\cong\PG(1,q)$ of
directions if each affine set that contains $\hu$ as a \emph{proper}
subset determines \emph{more than} $|\hd|$ directions.
\end{defi}
\textrm{Tam\'as Sz\H{o}nyi} proved a `completing theorem'
(stability result) in~\cite{Szonyi1996}, which was
slightly generalized in~\cite{Sziklai1999} as follows.
\begin{thm}[Sz\H{o}nyi; Sziklai]
{\rm \cite[Theorem~3.1]{Sziklai1999}}
Let $\hd$ denote the set of directions determined by the affine set
$\hu\subset\AG(2,q)$ containing $q-\varepsilon$ points, where
$\varepsilon<\alpha\sqrt{q}$ and $|\hd| < (q+1)(1-\alpha)$,
$1/2<\alpha<1$. Then $\hu$ can be extended to a set $\hu'$ with
$|\hu'|=q$ such that $\hu'$ determines the same directions.
\qed
\end{thm}
\textrm{Sz\H{o}nyi}'s
stability theorem above also stimulates us to restrict ourselves to
examine the \emph{maximal} affine sets only. (An affine set of $q$
points that does not determine all directions is \emph{automatically}
maximal.)

If we examine polynomials in one variable instead of R\'edei
polynomials, we can get similar `stability' results.
Such polynomials occur when we examine $R(X,y)$, $Q(X,y)$ and
$H(X,y)$, or $R$, $Q$ and $H$ over $\GF(q)(Y)$.
The second author conjectured that if `almost all' roots of a
polynomial $g\in\GF(q)[X]$ have muliplicity a power of $p$ then the
quotient $X^q\div g$ extends $g$ to a polynomial in $\GF(q)[X^p]$.
We can prove more.
\begin{nota}
Let $p=\kar\uf\neq 0$ be the characteristic of the \emph{arbitrary}
field $\uf$. Let $s=p^e$ and $q=p^h$ two arbitrary powers of $p$ such
that $e\le h$ (i.e. $s\mid q$ but $q$ is not necessarily a power of
$s$).
\end{nota}

\begin{thm}\label{thm:egydim}
Let $g,f\in\uf[X]$ be polynomials such that $g\cdot f\in\uf[X^s]$.
If $0\le\deg f\le s-1$ then $X^q\div g$ extends $g$ to a polynomial
in $\uf[X^s]$.
\end{thm}
\begin{proof}
We know that $\deg(gf)=ks$ ($k\in\mathbb{N}$).
Let $r=X^q\mod (fg)$ denote the remainder, that is, $X^q=(gf)h +
r$ where $\deg r\le\deg(fg)-1$.

\noindent
Now we show  that $h\in\uf[X^s]$. Suppose to the contrary that
$h\notin\uf[X^s]$, i.e. the polynomial $h$ has at least one
monomial $\bar{a}\notin\uf[X^s]$. Let $\bar{a}$ denote such a monomial
of \emph{maximal degree}. Let $\bar{b}$ denote the leading term of
$fg$. Since $gf\in\uf[X^s]$, also $\bar{b}\in\uf[X^s]$, and since
$\bar{b}$ is the leading term, $\deg\bar{b}=\deg(fg)$. The product
$\bar{a}\bar{b}$ is a monomial of the polynomial $(fg)h$, and since
$\deg\bar{b}>\deg r$, then $\bar{a}\bar{b}$ is also a monomial of the
polynomial $(f\cdot g\cdot h+r)$. The monomial $\bar{a}\bar{b}$ is not
in $\uf[X^s]$, because $\bar{a}\notin\uf[X^s]$ and
$\bar{b}\in\uf[X^s]$. But $f\cdot g\cdot h+r=X^q\in\uf[X^s]$, which
is a contradiction.

\noindent
Hence $r\in\uf[X^s]$, and so $s\mid\deg r$, and thus, if the closed
interval $[\deg g,\linebreak[0]\deg(gf)-1]$ does not contain any
integer that is a multiple of $s$ then $\deg r$ is less than $\deg g$.

\noindent
If we know that $\deg r<\deg g$ then from the equation
$X^q=(gf)h + r$ we get $X^q=g(fh) + r$ hence $r=X^q\mod g$ and
$X^q \div g = f\cdot h$, where $f$ is a polynomial such that
$fg\in\uf[X^s]$ and also $h\in\uf[X^s]$.

\noindent
So it is enough to show that the closed interval $[\deg g,\deg(gf)-1]$
does not contain any integer that is a multiple of $s$. Using
$\deg(fg)=k s$, the closed interval
$[\deg g,\deg(gf)-1]=[k s-\deg f,k s-1]$ and if $0\le\deg f\le s-1$
then it does not contain any integer which is the multiple of $s$.
\end{proof}

This theorem above suggests that if the product
$R(X,y)Q(X,y)$ is an element of $\GF(q)[Y][X^{p\cdot s(y)}]$ while
$R(X,y)\in\GF(q)[Y][X^{s(y)}]\setminus\GF(q)[Y][X^{p\cdot s(y)}]$,
then a `completing result' might be in the background.
If $\hu$ is a maximal affine set then it cannot be completed, so
we conjecture the following.
\begin{conj}
If $\hu$ is a \emph{maximal} affine set that determines the set $\hd$
of directions then $t(y)=s(y)$ for all $y\in\hd$ where $t(y)>2$.
\end{conj}

Note that there can be maximal affine sets which are not linear.
\begin{ex}[Non-linear maximal affine set]
Let $\hu\subset\AG(2,q)$ be a set, $|\hu|=q$, $s=1$,
$q\ge|\hd|\ge\dfrac{q+3}{2}$. In this case $\hu$ cannot be linear
because then $s$ would be at least $p$. But $\hu$ must be maximal
since $q+1$ points in $\AG(2,q)$ would determine all directions.
Embed $\AG(2,q)$ into $\AG(2,q^m)$ as a subgeometry. Then
$\hu\subset\AG(2,q^m)$ is a maximal non-linear affine set of less than
$q^m$ points.
\end{ex}

But if $s>2$, we conjecture that the maximal set is linear.
\begin{conj}
If $\hu$ is a \emph{maximal} affine set that determines the set $\hd$
of directions and $t=s>2$ then $\hu$ is an affine $\GF(s)$-linear set.
\end{conj}

Although we conjecture that the maximal affine sets with $s=t>2$ are
linear sets, the converse is not true.
\begin{ex}[Non-maximal affine linear set]
Let $\AG(2,s^i)$ be a canonical subgeometry of $\AG(2,q=s^{i\cdot j})$
and let $\hu$ be an affine $\GF(s)$-linear set in the subgeometry
$\AG(2,s^i)$ that contains more than $s^i$ points. Then $\hu$
determines the same direction set that is determined by the
subgeometry $\AG(2,s^i)$.
\end{ex}

\end{document}